\documentclass[a4paper]{amsart}
\usepackage{amsfonts}
\usepackage{amsthm}
\usepackage{amssymb}

\usepackage[centertags]{amsmath}
\usepackage[all]{xy}

\def\0{{\mathbbm 0}}
\def\1{{\mathbbm 1}}

\def\b{\mathbf}
\def\ba{\begin{array}}

\def\bcd{\begin{CD}}
\def\bea{\begin{eqnarray}}
\def\beaa{\begin{eqnarray*}}
\def\beq{\begin{equation}}

\def\bma{\begin{matrix}}

\def\bpm{\begin{pmatrix}}

\def\c{\centerline}
\def\C{\mathbb{C}}

\def\De{\Delta}

\def\dim{\textrm{dim}\,}

\def\ea{\end{array}}
\def\ecd{\end{CD}}
\def\eea{\end{eqnarray}}
\def\eeaa{\end{eqnarray*}}
\def\eeq{\end{equation}}
\def\ema{\end{matrix}}

\def\epm{\end{pmatrix}}

\def\F{\mathbb F}

\def\lan{\langle}

\def\mr{\mathrm}

\def\O{\mathcal{O}}

\def\ol{\overline}
\def\om{\omega}

\def\ov{\over}

\def\Pic{\mathrm{Pic}}

\def\Proj{\mathbb P}

\def\ran{\rangle}

\def\si{\sigma}

\def\t{\tilde}

\def\Th{\Theta}

\def\today{\ifcase\month\or January\or February\or  March\or  April\or
May\or June\or July\or August\or  September\or October\or
November\or December\fi  \space\number\day, \number\year}

\def\Z{\mathbb Z}

\def\sqr#1#2{{\vcenter{\hrule height .#2pt
      \hbox{\vrule width .#2pt height#1pt \kern#1pt\vrule width.#2pt}
                        \hrule height.#2pt}}}

\theoremstyle{plain}
\newtheorem{thm}{Theorem}[section]
\newtheorem*{thm*}{Theorem}

\newtheorem{prop}[thm]{Proposition}
\newtheorem{cor}[thm]{Corollary}

\theoremstyle{definition}

\title{Genus three curves and 56 nodal sextic surfaces}
\author{Bert van Geemen, Yan Zhao}
\address{Dipartimento di Matematica, Universit\`a di Milano, via Saldini 50, 20133 Milano, Italia}
\email{lambertus.vangeemen@unimi.it}
\address{Dipartimento di Matematica, Universit\`a di Milano, via Saldini 50, 20133 Milano, Italia\newline
Mathematisch Instituut, Universiteit Leiden, Niels Bohrweg 1, 2333CA Leiden, The Netherlands}
\email{y.zhao@math.leidenuniv.nl}
\begin{document}
\begin{abstract}
Catanese and Tonoli showed that the maximal cardinality for an even set of nodes
on a sextic surface is 56 and they constructed such nodal surfaces.
In this paper we give an alternative, rather simple, construction for these
surfaces starting from a non-hyperelliptic genus three curve. We illustrate
our method by giving explicitly the equation of such a sextic surface starting
from the Klein curve.
\end{abstract}
\maketitle

\section*{Introduction}
A nodal surface is a projective surface with only ordinary double points as
singularities.
A set of nodes of a surface $F$ is said to be even if
there is a double cover $S\rightarrow F$ branched exactly in the nodes from
that set.
In \cite{catanese-tonoli}, Catanese and Tonoli showed that an even set of
nodes on a sextic surface has cardinality in $\{24,32,40,56\}$.
They also provided a construction of such $56$ nodal surfaces,
constructions for the other cases were already known. Their method is based
on the paper \cite{casnati-catanese}, where it is shown that even sets of
nodes
correspond to certain symmetric maps between vector bundles. 
A careful study of the sheaves involved leads to certain  matrices
whose entries are homogeneous polynomials on $\Proj^3$. 
The points in $\Proj^3$
where such a matrix has rank less than $6$ is a sextic surface with an even set
of $56$ nodes. In this way, one can find explicit examples of such surfaces, but
the equations tend to be rather complicated and it is not easy to understand
the geometry of these surfaces.

Let $F$ be a $56$ nodal surface as constructed in \cite{catanese-tonoli}
and let $f:S\to F$ be the double cover which is branched exactly over the
nodes
of $F$. The first Betti number of the smooth surface $S$ is equal to $6$,
hence $S$ has a $3$-dimensional Albanese variety.
With some trial and error, we then found the following construction for $56$
nodal surfaces.  A principally polarized abelian threefold $A$ has a theta
divisor $\Th$,
defining the polarization, which can be taken to be symmetric, so
$[-1]\Th=\Th$.
The fixed points of $[-1]$ are the two-torsion points. There are exactly $28$
such points on $\Th$ precisely in the case that $(A,\Th)$ is the Jacobian of a
non-hyperelliptic genus three curve. Assume that we are in this case.
Then $\ol{\Th}=\Th/[-1]$ has $28$ nodes. This singular surface has
been studied before, cf.\  \cite[Chapter IX.6, Theorem 4 and Remark 6]{dolgachev-ortland}.
In particular it has an embedding into $\Proj^6$ where it is a Cartier divisor
in a cone over a Veronese surface. This cone is the quotient of $\Proj^3$
by an involution which changes the sign of one of the homogeneous coordinates.
The inverse image of $\ol{\Th}$ in $\Proj^3$ is then a sextic surface $F$ with
an even
set of $56$ nodes, as we show in Section \ref{sec_construction}.

By construction, $F$ has an involution with quotient $\ol{\Th}$.
This involution lifts to $S$ and together 
with the covering involution of the map $S\to F$
generates a subgroup $(\Z/2\Z)^2$ of $\mr{Aut}(S)$. 
In Section \ref{covers} we study the cohomology
of the quotients of $S$. 
We also show there that our construction and the one from Catanese and Tonoli
produce the same surfaces.
In the last section we give an explicit example, with a simple equation, 
of such a surface.

%%%%%%%%%%%%%%%%%%%%%%%%%%%%%%%%%%%%
%%%%%%%%%%%%%%%%%%%%%%%%%%%%%%%%%%%%
\section{Construction of a family of 56 nodal sextic
surfaces}\label{sec_construction}
Let $C$ be a smooth non-hyperelliptic curve of genus 3 and consider its
Jacobian  $A=\mr{Jac}(C)$. The abelian variety $A$ admits a principal
polarization defined by the theta divisor $\Th$ and we will identify $\Th=S^2C$.
We can choose $\Th$ to be a symmetric divisor on $A$, i.e. $[-1]^*\Th=\Th$.
The involution $[-1]$ on $\Theta$ corresponds to the involution $D\mapsto K_C-
D$ on $S^2C$, where $K_C$ is the canonical divisor on $C$.

The linear system $|2\Th|$ is totally symmetric, and defines a morphism
$$
\varphi_{2\Th}:A
\,\longrightarrow\,\Proj^7
$$
which is the quotient map by the involution $[-1]$. Let $\ol{A}\cong A/[-1]$,
the Kummer variety of $A$, be the image of $\varphi_{2\Th}$.
The singular locus of $\ol{A}$ consists of 64 nodes, these are the images of
the two-torsion points of $A$.

Consider the hyperplane $H_{2\Th}$ of $\Proj^7$ corresponding to the divisor
$2\Th$.
The intersection of $H_{2\Th}$ with $\ol{A}$ is the image
$\ol{\Th}\cong\Th/[-1]$
of $\Th$, with multiplicity two.
As $\Th$ contains $28$ of the two-torsion points of $A$, the surface
$\ol{\Th}$ has 28 nodes. Equivalently, 
these are the images of the $28$ odd theta characteristics in $S^2C$.

To describe this map $\varphi_{2\Th}|_{\Th} :\Th\to \Proj^6$, notice that
the adjunction formula on $A$ shows that the canonical class of $\Th$ is
$K_\Th=\Th_{|\Th}$. Thus $\O_\Th(2\Th)\cong \omega_\Th^{\otimes 2}$.
Moreover, the cohomology of the restriction sequence
$$
0\,\longrightarrow\,\O_A(\Th)\,\longrightarrow\,
\O_A(2\Th)\,\longrightarrow\,\O_\Th(2\Th)\,\longrightarrow\,0
$$
combined with $H^i(A,\O_A(\Th))=0$ for $i>0$ (Kodaira vanishing or Riemann-
Roch on $A$), shows that
$h^0(\om_\Th^{\otimes 2})=h^0(\O_\Th(2\Th))=7$.
Hence, when restricted to $\Th$, 
the morphism $\varphi_{2\Th}|_\Th=\varphi_{2K_\Th}$
is given by the complete linear system $|2K_\Th|$.

To understand this morphism better, we first consider the map
$\varphi_{K_\Th}$.
From the restriction sequence above, twisted by $\O_A(-\Th)$, one deduces that
$H^0(\Th,\omega_\Th)\cong H^1(A,\O_A)$ is three dimensional. 
The map $\varphi_{K_\Th}:\Th\rightarrow \Proj^2$ is the Gauss map, which is a
morphism of degree $(\Th_{|\Th})^2=\Th^3=6$ which factors over $\ol{\Th}$.
As $\varphi_{K_\Th}$ is surjective,
the natural map $S^2H^0(\Th,\omega_\Th)\rightarrow H^0(\Th,\omega_\Th^{\otimes
2})$
is injective, thus the image has codimension one.

Let $t\in H^0(\Th,\omega_\Th^{\otimes 2})$ be a general section in the
complement of
the image of $S^2H^0(\Th,\omega_\Th)$.
Since $|2\Th|$ is basepoint free, we may assume that the divisor $B$ in $\Th$
defined by $t=0$ is smooth and does not pass through any two-torsion points.
Since $|2\Th|$ is totally symmetric, any divisor in this linear system is
symmetric,
that is, $[-1]B=B$.
Let $s_0,\ldots,s_2$ be a basis of $H^0(\Th,\omega_\Th)$. Then we have:
$$
\varphi_{2K_\Th}:\;\Th\,\longrightarrow\,
\Proj H^0(\Th,\omega_\Th^{\otimes 2})\,\cong\,H_{2\Th}\,\cong\,\Proj^6\,~,
\qquad
x\,\longmapsto\,(\ldots:s_i(x)s_j(x):\ldots:t(x))_{0\leq i\leq j\leq 2}~.
$$
The image $\ol{\Th}$ of $\Th$ thus lies in a cone over the Veronese surface
of $\Proj^2$. This cone is the image $Y$ of the weighted projective 3-space
$\Proj(1,1,1,2)$, which is embedded into $\Proj^6$ by the (very) ample
generator
$\O_Y(1)$ of its Picard group:
$$
\Proj(1,1,1,2)\,\longrightarrow\,Y\,\subset\,\Proj^6,\qquad
(y_0:y_1:y_2:y_3)\,\longmapsto\,(\ldots:y_iy_j:\ldots:y_3)_{0\leq i\leq j\leq
2}~.
$$
As $\varphi_{K_\Th}$ has no base points, the surface $\ol{\Th}\subset Y$ does
not
contain the singular point $v=(0:\ldots:0:1)$ of $Y$, the vertex of the cone
over the Veronese surface. Hence, $\ol{\Th}$ is a Cartier divisor on $Y$.
The projection of $\ol{\Th}$ from $v$ onto the Veronese surface is the Gauss
map $\varphi_{K_\Th}$, which has degree $6/2=3$ on $\ol{\Th}$.
This implies that $\ol{\Th}$ lies in the linear system on $Y$ defined by three
times
the ample generator. Since the map $S^3H^0(Y,\O_Y(1))\rightarrow H^0(Y,
\O_Y(3))$
is surjective, we conclude that $\ol{\Th}$ is defined by a weighted homogeneous
polynomial $p$ of degree six in $Y=\Proj(1,1,1,2)$:
$$
\ol{\Th}\,=\,\{(y_0:y_1:y_2:y_3)\in \Proj(1,1,1,2):\;
p(y_0,\ldots,y_3)\,=\,\sum_{i=0}^3 p_{2i}(y_0,y_1,y_2)y_3^{3-i}\,=\,0\,\}~,
$$
where each $p_{2i}$ is homogeneous of degree $2i$ in $y_0,y_1,y_2$.
Since $v\not\in \ol{\Th}$, we may and will assume that $p_0=1$.

The weighted projective space $\Proj(1,1,1,2)$ is also the quotient of
$\Proj^3$
by the involution $i_3:(x_0:\ldots:x_3)\mapsto (x_0:x_1:x_2:-x_3)$,
the quotient map is explicitly given by:
$$
\ol{p}:\,\Proj^3\,\longrightarrow\,\Proj(1,1,1,2),\qquad
(x_0:x_1:x_2:x_3)\,\longmapsto\,(\ldots:x_ix_j:\ldots:x_3^2)_{0\leq i\leq
j\leq 2}~.
$$

Now we define a surface $F$ in $\Proj^3$ as $F:=\ol{p}^{-1}(\ol{\Th})$, thus
$F$ is
defined by the sextic equation $P=0$ where
$$
P\,:=\,  p_{6}(x_0,x_1,x_2)\,+\,p_4(x_0,x_1,x_2)x_3^2
\,+\,p_2(x_0,x_1,x_2)x_3^4\,+\,x_3^6~.
$$
The double cover $\ol{p}:F\to\ol{\Th}$ is branched over the points where
$x_3=0$,
so the branch locus is the divisor $\ol{B}\subset\ol{\Th}$ defined by $t=0$.
Here $\ol{B}=B/[-1]$, which is a smooth curve since by assumption $B$
is smooth and does not pass through the $28$ fixed points of $[-1]$ in ${\Th}$.
Hence the singular locus of $F$ consists of $56$ nodes.
The 28 nodes of $\ol{\Th}$ form an even set since the
double cover $\Th\rightarrow \ol{\Th}$ is branched only over the nodes.
Hence the preimage $\De\subset F$ of these nodes is also an even set, cf.\
diagram
\ref{cd_main}, in fact $F$ has a double cover $S$ branched only over the nodes
by pulling back the double cover $\Th\rightarrow \ol{\Th}$ along
$\ol{p}:F\rightarrow\ol{\Th}$.

We summarize the construction as follows:

\begin{thm}\label{ours}
There exists a family of 56 nodal sextic surfaces with the nodes forming an
even set, which is parametrized by pairs $(C,B)$ where $C$ is a non-hyperelliptic curve of
genus $3$ and $B\in |2K_{S^2C}|$ is a general divisor.
In particular, we have a $6+6=12$ dimensional family of such surfaces.
Moreover, each surface in the family has an automorphism of order two.
\end{thm}

\section{Coverings of $\ol{\Th}$}\label{covers}

First of all, we provide another construction of the double cover $f:S\to F$
branched over the set $\De$ of 56 nodes of $F$. 
Let $\pi_F:\t F\to F$ be a minimal resolution of singularities and let  
$N_i\subset\t F$ be the inverse image of the node $p_i\in F$. 
Since the nodes form an even set, 
the divisor $\t\De=\sum_{i=1}^{56}N_i$ is even, that is, 
it is 2-divisible in $\Pic(\t F)$.
Thus $\t F$ admits a smooth double cover $\t f:\t S\to\t F$ branched along $\t\De$.
Let $E_i=\t f^{-1}(N_i)$, so $\t f^*N_i=2E_i$. Since $p_i$ are nodes, 
the exceptional curves $N_i$ are $(-2)$-curves, so
$$
E_i\cdot E_i\,=\,{1\ov 4}\t f^*N_i\cdot\t f^*N_i\,=\,{1\ov 2}\t f^*(N_i\cdot N_i)
\,=\,-1
$$
and $E_i$ are $(-1)$-curves. 
The surface $S$ can be obtained by blowing down this set of $(-1)$-curves on the smooth surface $\t S$, so it is also smooth and it is a double cover of $F$,
giving a commutative diagram
$$
\xymatrix{\t{S}\ar[r]^-{\pi_S}\ar[d]_{\t{f}}&S\ar[d]^f\\\t{F}\ar[r]_-{\pi_F}&F~.}
$$

From the definition of $S$ as the base change along $\ol{p}:F\rightarrow\ol{\Th}$
of the double cover $\Th\rightarrow \ol{\Th}$, it follows that 
the covering $S\rightarrow \ol{\Th}$ is a $(\Z/2\Z)^2$-covering. 
Let $\iota_1$ and $\iota_2$ be involutions on $S$ with quotient surface $F$ 
and $\Th$ respectively. Let $\iota_3=\iota_1\iota_2$, then $\iota_3$ is an involution
and we define $T:=S/\iota_3$.

This gives a commutative diagram
\bea\label{cd_main}\xymatrix{S\ar[rr]^-{p}\ar[dd]_-{f}\ar[rd]&&\Th\ar[dd]^-{\phi}\\
&T\ar[rd]&\\
F\ar[rr]_-{\ol{p}}&&\ol{\Th}}\eea

\begin{prop}
The double cover $S\to T$ is unramified. In particular, $T$ is smooth and $T\to\ol{\Th}$ is branched along $\ol B$ and the 28 nodes.
\end{prop}
\begin{proof}
The ramification locus of $S\to T$ is the fixed locus $R_3$ of $\iota_3$, which is precisely the points $s\in S$ such that $\iota_3\in\mr{Stab}_{(\Z/2\Z)^2}(s)$. The fixed loci of $\iota_1$ and $\iota_2$ are $f^{-1}\De$ and $p^{-1}B$ respectively. Since the branch curve $B$ does not contain any of the $28$ two-torsion points on $\Theta$, the intersection of the fixed loci $f^{-1}\De\cap p^{-1}B=\emptyset$. Hence, there are no points $s\in S$ such that $\mr{Stab}_{(\Z/2\Z)^2}(s)=(\Z/2\Z)^2$. In particular, $R_3=\{s\in S|\mr{Stab}_{(\Z/2\Z)^2}(s)=\lan\iota_3\ran\}$
is disjoint from $f^{-1}\De\cup p^{-1}B$. 
Since the ramification locus of $S\to\ol\Th$ is precisely the union of 
that of $f$ and $p$, we conclude that $R_3=\emptyset$ and $S\to T$ is unramified. 
\end{proof}

We now consider the Hodge numbers of the surfaces in diagram \ref{cd_main}.

\begin{prop}
The smooth surfaces $\Th$, $S$, $T$ and $\t F$ have Hodge numbers:
$$
\begin{array}{cccc}
 &h^{1,0}&h^{2,0}&h^{1,1}\\
\Th&3&3&10\\
S&3&10&38\\
\t S&3&10&94\\
\t F&0&10&86\\
T&0&3&16
\end{array}
$$
\end{prop}
\begin{proof}
The cohomologies of $\O_\Th$ are computed using the short exact sequence
$$
0\longrightarrow\O_A(-\Th)\longrightarrow\O_A\longrightarrow\O_\Th\longrightarrow0~.
$$
Since $\Th$ is ample on $A$ and $K_A=0$,  $h^i(\O_A(-\Th))=0$ for $i<3$ 
by Kodaira vanishing
and, using Serre duality, $h^3(\O_A(-\Th))=h^0(\O_A(\Th))=1$ since $\Th$ is a principal polarization. 
Moreover, $h^i(\O_A)=({}^3_i)$, hence $h^{1,0}(\Th)=h^1(\O_\Th)=3$ and
$h^{2,0}(\Th)=3$. 
As 
$$
\chi_{\mr{top}}(\Th)\,=\,2-4h^{1,0}(\Th)+2h^{2,0}(\Th)+h^{1,1}(\Th)\,=\,
2-12+6+h^{1,1}(\Th)\,=\,h^{1,1}(\Th)-4~,
$$
we can compute $h^{1,1}(\Th)$ from Noether's formula 
$$
\chi(\O_\Th)\,=\,{\chi_{\mr{top}}(\Th)+K_\Th^2\ov 12}\;\Rightarrow\; h^{1,1}(\Th)=12\chi(\O_\Th)-K_\Th^2+4\,=\,12-6+4\,=\,10~.
$$

For the double cover $p:S\to\Th$, branched over the divisor $B$, there is
an isomorphism
$$
p_*\O_S=\O_\Th\oplus\mathcal{L}^{-1},\qquad 
{\mathcal L}\,\cong\,\om_\Th~,
$$
so ${\mathcal L}^{\otimes 2}=\O_\Th(B)$. 
Thus $h^{i,0}(S)=h^{i}(\O_\Th)+h^i({\mathcal L}^{-1})$.
As ${\mathcal L}=\om_\Th$ is ample, 
by Kodaira vanishing we get $h^i({\mathcal L}^{-1})=0$ for $i<2$. 
Hence, by Riemann-Roch
$$
h^2({\mathcal L}^{-1})=\chi(\om_\Th^{-1})=\chi(\O_\Th)+{K_\Th\cdot(K_\Th+K_\Th)\ov2}=7~.
$$
On the canonical bundles, we have an isomorphism $\om_S=p^*(\om_\Th\otimes{\mathcal L})$. Thus, $K_S^2=p^*(2K_\Th)^2=8K_\Th^2=48$ and we obtain $h^{1,1}(S)=38$ by Noether's formula.

The blowup $\pi_S:\t S\to S$ at 56 points does not change $h^{1,0}$ and $h^{2,0}$, and $h^{1,1}(\t S)=h^{1,1}(S)+56$.

Since $\pi_F:\t F\to F$ is a blowup at isolated rational singularities, we have $h^i(\O_{\t F})=h^i(\O_F)$. The latter can be computed using the short exact sequence
$$0\longrightarrow\O_{\Proj^3}(-6)\longrightarrow\O_{\Proj^3}\longrightarrow\O_F\longrightarrow 0~.$$
Since the singularity is canonical, we have $\om_{\t F}=\pi_F^*\om_F=\pi_F^*(\om_{\Proj^3}\otimes\O_{\Proj^3}(F))_{|F}=\pi_F^*\O_{\Proj^3}(2)_{|F}$ by the adjunction formula. Thus, $K_{\t F}^2=(\O(2)_{|F}\cdot\O(2)_{|F})=2\cdot2\cdot6=24$. By Noether's formula, we obtain $h^{1,1}(\t F)=86$.

Finally, we use the eigenspace decomposition of the cohomologies on $S$ to compute the Hodge numbers of $T$. Let $G=(\Z/2\Z)^2$. The $G$-action on $S$ induces a decomposition $$H^i(S,\O_S)=\bigoplus_{\chi\in G^*}H^i(S,\O_S)_\chi$$
where $G^*=\{1,\chi_1,\chi_2,\chi_3\}$ is the character group and $\chi_i$ is chosen such that, if $H_i$ is the stabilizer of $\iota_i$, then $G^*_{|H_i}=\{1,\chi_i\}$. Hence, 
$$h^i(\O_F)=h^i(\O_S)_1+h^i(\O_S)_{\chi_1},\quad h^i(\O_\Th)=h^i(\O_S)_1+h^i(\O_S)_{\chi_2},\quad h^i(\O_T)=h^i(\O_S)_1+h^i(\O_S)_{\chi_3}.$$
From the Hodge numbers $h^{1,0}$ and $h^{2,0}$ of $S$, $T$ and $\t F$ 
(notice $h^{i,0}(F)=h^{i,0}(\t F)$), we obtain that 
$$
h^{1,0}(S)_\chi=\begin{cases}3&\chi=\chi_2,\\0&\chi\ne\chi_2,\end{cases}
\qquad\text{and}\qquad 
h^{2,0}(S)_\chi=\begin{cases}3&\chi=1,\\7&\chi=\chi_1,\\0&\chi=\chi_2,\chi_3~.
\end{cases}
$$
Hence, $h^{1,0}(T)=0$ and $h^{2,0}(T)=3$. Since $S\to T$ is an unramified double cover, we have an equality $\chi_{\mr{top}}(S)=2\chi_{\mr{top}}(T)$. 
This allows us to compute $h^{1,1}(T)=16$.
\end{proof}

A consequence of the fact that we have a morphism $p:S\rightarrow\Th$ and
$h^{1,0}(S)=h^{1,0}(\Th)$, is that the Albanese map of $S$ factors over
the Albanese map for $\Th$, which is just the inclusion 
$\Th\hookrightarrow A$, hence $A=\mr{Alb}(S)$.

We now deduce that the $12$ dimensional family of $56$ nodal sextics we constructed
coincides with the family constructed by Catanese and Tonoli in \cite[Main Theorem B]{catanese-tonoli}.
Notice that they obtain a $27$ dimensional subvariety of the space of sextic surfaces
parametrizing $56$ nodal sextics, but modulo the action of $\mr{Aut}(\Proj^3)$ one again 
finds a $27-15=12$ dimensional family. We were not able to relate their construction
to ours. However, when using their Macaulay scripts (which can be found in 
the eprint arXiv:math/0510499) we noticed that it does produce sextics which are 
invariant under the involution $x_0\mapsto -x_0$ in $\Proj^3$.

\begin{cor}
The family of sextics with an even set of $56$ nodes from
\cite[Main Theorem B]{catanese-tonoli} coincides with the family constructed
in Theorem \ref{ours}.
\end{cor}

\begin{proof}
For a double cover
$f:S\rightarrow F$ of a $56$ nodal sextic surface $F$, branched 
exactly over the nodes of $F$, the `quadratic' sheaf ${\mathcal F}$ on $F$ defined by 
$f_*\O_S=\O_F\oplus{\mathcal F}$ must satisfy $(\tau,a)=(3,3)$ or $(\tau,a)=(3,4)$,
where $2\tau=h^1(F,{\mathcal F}(1))$ and $a=h^1(F,{\mathcal F})$, 
cf.\ \cite[Theorem 2.5]{catanese-tonoli}.
The family constructed in \cite{catanese-tonoli} is the one with invariants $(\tau,a)=(3,3)$.
For our surfaces we have $h^{1,0}(S)=h^1(F,f_*\O_S)=h^1(F,\O_F)+h^1({F,\mathcal F})$ 
so we get $h^1({F,\mathcal F})=3$, which shows that they are in the same family.
\end{proof}

%%%%%%%%%%%%%%%%%%%%%%%%%%%%%%%%%%
%%%%%%%%%%%%%%%%%%%%%%%%%%%%%%%%%%

\section{An explicit example}
Let $C$ be a non-hyperelliptic genus three curve, we will also denote the canonical
model of $C$, a quartic curve in $\Proj^2$, by $C$. Recall that $\Th=S^2C$, 
the symmetric product of $C$.

We show how to find the global sections
$H^0(\Th,\omega_\Th^{\otimes 2})$ in terms of the geometry of $C$, following 
\cite{brivio-verra}. 
Note that if we map $S^2C\rightarrow \mr{Jac}(C)$ by $p+q\mapsto p+q-t$ where $t\in S^2C$ is an odd theta characteristic (so $2t\equiv K_C$), 
then the image of $S^2C$ is a symmetric theta divisor.
Let $d=z_1+\ldots+z_4$ be an effective canonical divisor on $C$,
$D=\sum (z_i+C)$ be the corresponding divisor on $S^2C$ and $\Delta$ be the 
diagonal in $S^2C$. Then, $2K_{S^2C}=2D-\Delta$.
By \cite[Lemma 4.7]{brivio-verra}, we have the restriction sequence
$$
0\,\longrightarrow\,\O_{S^2C}(2K_{S^2C})\,\longrightarrow\,
\O_{S^2C}(2D)\,\longrightarrow\,\O_\Delta(2D)\cong\O_C(4d)\,\longrightarrow\,0~.
$$
and
$$
H^0(S^2C,\om_{S^2C}^{\otimes 2})\,\cong \, 
\ker\big(S^2H^0(C,\om_C^{\otimes 2})\,\stackrel{\mu}{\longrightarrow}\,H^0(C,\om_C^{\otimes 4})\big)~.
$$
where $\mu$ is the multiplication map.
Note that $h^0(C,\om_C^{\otimes 2})=6$, so $\dim S^2H^0(C,\om_C^{\otimes 2})=21$, and $h^0(C,\om_C^{\otimes 4})=14$.
By the same lemma, $\mu$ is surjective so indeed $h^0(S^2C,\om_{S^2C}^{\otimes 2})=7$.

Let $\sigma_0,\ldots,\sigma_2$ be a basis of $H^0(C,\om_C)$. It induces a basis
$\sigma_i\otimes\sigma_j$ of $H^0(C^2,\om_{C^2})=H^0(C,\om_C)^{\otimes 2}$.
The sections of $H^0(\Th,\omega_\Th)\cong \wedge^2H^0(C,\om_C)$ define the Gauss map $S^2C\cong\Theta\rightarrow \Proj^2$.
Explicitly, the Gauss map is induced by the map 
$$
C\times C\,\longrightarrow\, \Proj^2,\quad
(x,y)\,\longmapsto\,(p_{12}:p_{13}:p_{23}),\quad
p_{ij}(x,y):=\sigma_i(x)\sigma_j(y)-\sigma_j(x)\sigma_i(y)~.
$$ 
The six products $p_{ij}p_{kl}$ span a six dimensional subspace of
$\ker(\mu)$ which is the image of 
$S^2H^0(\Th,\omega_\Th)$ in $H^0(\Th,\omega_\Th^{\otimes 2})$.

Let $f(z)$ be a homogeneous quartic polynomial in $\C[z_0,z_1,z_2]$
such that $f(\si_0(x),\si_1(x),\si_2(x))=0$ for all $x\in C$, that is, 
$f$ defines the curve $C\subset\Proj^2$.
Choose any polynomial $g(u,v)$ of bidegree $(2,2)$ in $\C[u_0,u_1,u_2,v_0,v_1,v_2]$
such that $g(z,z)=f(z)$ and let $g_s(u,v):=g(u,v)+g(v,u)$, then 
$\tilde{g}(x,y):=g_s(\si_0(x),\ldots,\si_2(y))\in S^2H^0(C,\om_C^{\otimes 2})$ and lies in $\ker(\mu)$.
Thus the choice of $\tilde{g}$ provides the section $t$ used to construct the map 
$\varphi_{2K_\Th}$, any other choice of $\tilde{g}$ is of the form
$\lambda\tilde{g}+\sum \lambda_{ij}p_{ij}$ for complex numbers $\lambda,\lambda_{ij}$
with $\lambda\neq 0$.

The map $\varphi_{2K_\Th}:\Th\rightarrow\Proj^6$ is therefore induced by the map
$$
C\times C\,\longrightarrow\, \Proj^6,\qquad
(x,y)\,\longmapsto\,\big(\cdots:p_{ij}(x,y)p_{kl}(x,y):\cdots:\tilde{g}(x,y)\big)~.
$$
A homogeneous polynomial $P$ in seven variables is an equation for the image of this map
if \\
$P(\ldots,p_{ij}(u,v)p_{kl}(u,v),\ldots,\tilde{g}(u,v))$ lies in the ideal of 
$\C[u_0,\ldots,v_2]$ generated by $f(u)$ and $f(v)$. 

\

An explicit example, worked out using the computer program Magma \cite{magma}, 
is provided by the choice $f=z_0z_1^3+z_1z_2^3+z_2z_0^3$, 
which defines the Klein curve in $\Proj^2$. 
We will take $g=u_0u_1v_1^2+u_1u_2v_2^2+u_2u_0v_0^2$ and the map $\varphi_{2K_\Th}$
is given by:
$$
(y_{00}:y_{01}:\ldots:y_{22}:y_g)\,=\,
\big(p_{01}^2:p_{01}p_{02}:p_{01}p_{12}:p_{02}^2:p_{02}p_{12}:p_{12}^2:
\tilde{g}\big)~.
$$
One of the equations for the image is
$$
y_{00}^2y_{02} -y_{12}y_{22}^2  -y_{01}y_{11}^2 -5y_{01}^2y_{22}+
(-y_{00}y_{01}+ y_{02}y_{22} -y_{11}y_{12})y_g  -y_g^3
$$
(this equation thus defines the image in $\Proj(1,1,1,2)\subset\Proj^6$). 
Next we pull this equation back to $\Proj^3$ along the map $\ol{p}$ by
substituting $y_{ij}=x_ix_j$ and $y_g=x_3^2$, {\it moreover} we change the sign of
$x_1$ in order to simplify the equation and we obtain
$$
Q\,:=\,
x_0^5x_2 + x_0x_1^5 + x_1x_2^5
 - 5x_0^2x_1^2x_2^2
\,+\,( x_0^3x_1 + x_0x_2^3 + x_1^3x_2)x_3^2  - x_3^6~.
$$
The singular locus of the surface $F$ defined by $Q=0$ consists of $56$ nodes
and these are thus an even set of nodes. To find all the nodes, we observe that
$\mr{Aut}(F)$ contains a subgroup $G_{336}$ of order $336$ with generators
$$
g_7\,:=\,\mbox{diag}(\omega,\omega^4,\omega^2,1),\qquad
g_2:=\frac{1}{\sqrt{-7}}\left(\begin{array}{cccc} a&c&b&0\\c&b&a&0\\b&a&c&0\\0&0&0&\sqrt{-7}\end{array}\right)~,\qquad
\left\{\begin{array}{rcl} a&=&\omega^2-\omega^5,\\
        b&=&\omega-\omega^6,\\
        c&=&\omega^4-\omega^3~,
       \end{array}\right.
$$
where $\omega$ is a primitive seventh root of unity. 
One of the nodes is $(1:1:1:1)$ and 
$G_{336}$ acts transitively on the $56$ nodes, the stabilizer of a node is isomorphic
to the symmetric group $S_3$. The covering involution $\mr{diag}(1,1,1,-1)$ generates the center of $G_{336}$ and $G_{336}\cong \{\pm 1\}\times G_{168}$ where $G_{168}\cong
SL(3,\F_2)$ is the automorphism group of the Klein curve.
The equation of $F$ can be written as $p_6+p_4x_3^2-x_3^6$, the discriminant of the cubic polynomial $p_6+p_4T-T^3$ has degree $12$ in $\C[x_0,x_1,x_2]$
and the curve it defines is the dual of the Klein curve (as expected from the
presence of the Gauss map).

\bibliographystyle{alpha}

\end{document}